\documentclass[preprint,11pt]{elsarticle}
\usepackage{amsmath,amssymb,amsthm,latexsym}
\usepackage{graphicx}
\renewenvironment{proof}{\noindent{\bf Proof.}}{~~$\Box$}
\theoremstyle{plain}
\newtheorem{thrm}{Theorem}[section]

\theoremstyle{definition}
\newtheorem{dfnt}[thrm]{Definition}
\newtheorem{remk}[thrm]{Remark}
\newtheorem{exa}[thrm]{Example}

\newtheorem{obs}{Observation}
\newcommand{\Z}{{\mathbb{Z}}}

\begin{document}
\journal{}
\begin{frontmatter}
\title{{\bf\large Ordinal sums of triangular norms on a bounded lattice}}
\author[1]{Yao Ouyang}
\ead{oyy@zjhu.edu.cn}
\author[2]{Hua-Peng Zhang\corref{cor2}}
\ead{huapengzhang@163.com}
\author[3]{Bernard De Baets}
\ead{Bernard.DeBaets@Ugent.be}
\cortext[cor2]{Corresponding author.}
\address[1]{Faculty of Science, Huzhou University,
            Huzhou, Zhejiang 313000, China}
\address[2]{School of Science, Nanjing University of Posts and Telecommunications,
Nanjing 210023, China}
\address[3]{KERMIT, Department of Data Analysis and Mathematical Modelling, Ghent University, Coupure links 653, 9000 Gent, Belgium}
\begin{abstract}
The ordinal sum construction provides a very effective way to generate a new triangular norm on the real unit interval
from existing ones. One of the most prominent theorems concerning the ordinal sum of triangular norms on the real unit interval
states that a triangular norm is continuous if and only if it is uniquely representable as an ordinal sum of continuous Archimedean triangular norms. However, the ordinal sum of triangular norms on subintervals of a bounded lattice is not always a triangular norm (even if only one summand is involved), if one just extends the ordinal sum construction to a bounded lattice in a na\"{\i}ve way. In the present paper, appropriately dealing with those elements that are incomparable with the endpoints of the given subintervals, we propose an alternative definition of ordinal sum of countably many (finite or countably infinite) triangular norms on subintervals of a complete lattice, where the endpoints of the subintervals constitute a chain. The completeness requirement for the lattice is not needed when considering finitely many triangular norms. The newly proposed ordinal sum is shown to be always a triangular norm. Several illustrative examples are given.
\end{abstract}

\end{frontmatter}

\section{Introduction}
The ordinal sum construction provides a method to construct a new semigroup from existing ones~\cite{Cli54}.
Ling~\cite{Lin65} and Schweizer and Sklar~\cite{SchSkl63} applied this method to a special kind of semigroup, namely
to triangular norms (t-norms, for short) on the real unit interval~$[0, 1]$. One of the most prominent theorems concerning
the ordinal sum of t-norms states that a t-norm is continuous if and only if it is uniquely representable as an ordinal sum
of continuous Archimedean t-norms (see, e.g.,~\cite{AlsFraSch06,KleMesPap00}).

T-norms on more general structures (e.g., posets~\cite{CooKer94,Zha05} and bounded lattices~\cite{BaeMes99,Dro99}) have been proposed and extensively investigated. In 2006, Saminger~\cite{Sam06} extended the ordinal sum of t-norms on the real unit interval~$[0, 1]$ to the ordinal sum of t-norms on subintervals of a bounded lattice  in a rather direct way without much consideration for the characteristics of a lattice, especially the existence of elements that are incomparable with the endpoints of the given subintervals. Unfortunately, Saminger's ordinal sum of t-norms on subintervals of a bounded lattice does not always yield a t-norm even in the case of a single summand. Some researchers~\cite{Med12,SamKleMes08} characterized when Saminger's ordinal sum of t-norms always leads to a t-norm, while other researchers attempted to modify Saminger's ordinal sum or considered the ordinal sum problem for a particular class of lattices. For instance, Ertu\u{g}rul et al.~\cite{ErtKarMes15} modified Saminger's ordinal sum for one special summand to make sure it results in a t-norm. El-Zekey~\cite{Elz19} studied the ordinal sum of t-norms on bounded lattices that can be written as a lattice-based sum of lattices. Up to now, the ordinal sum problem has not yet been solved completely.

Although Saminger's definition of ordinal sum of t-norms on a bounded lattice is a natural extension of the ordinal sum of t-norms on the real unit interval~$[0, 1]$, it is not satisfactory since it does not always lead to a t-norm. This motivates the following question:

{\it Does there exist a more appropriate definition of ordinal sum of t-norms on a bounded lattice}?

We argue that any definition of ordinal sum of t-norms on a bounded lattice that reduces to the ordinal sum of t-norms on $[0, 1]$ could be a possible candidate for the answer to the above question. The key lies in whether it always leads to a t-norm. In this paper, appropriately dealing with those elements that are incomparable with the endpoints of the given subintervals, by synthesizing the techniques of~\cite{BodKal14} and~\cite{ErtKarMes15}, we propose an alternative definition of ordinal sum of countably many t-norms on subintervals of a complete lattice, where the endpoints of the subintervals constitute a chain. The completeness requirement for the lattice is not needed when considering finitely many t-norms. Our proposed ordinal sum is shown to be always a t-norm.

Admittedly, a t-norm on a bounded lattice is nothing else but a commutative and integral partially ordered monoid (pomonoid, for short)~\cite{Hohle95} and the ordinal sum of pomonoids has been investigated in the literature~\cite{Fuchs63}. However, the ordinal sum of pomonoids is different from the above-mentioned ordinal sum of t-norms since the former is defined on the direct sum of the underlying posets, while the latter is defined on a bounded lattice that is not necessarily the direct sum of the underlying subintervals of the lattice. The poset product of pomonoids proposed in~\cite{JipMon09} generalizes both the ordinal sum of pomonoids and the direct product of pomonoids. Concretely speaking, the poset product of pomonoids reduces to the ordinal sum of pomonoids when the underlying index set is a chain, while it reduces to the direct product of pomonoids when the underlying index set is an antichain. Therefore, the poset product of pomonoids cannot cover our ordinal sum of t-norms due to the fact that in our ordinal sum of t-norms the underlying index set is a chain and the above-mentioned difference between the ordinal sum of pomonoids and the ordinal sum of t-norms.

The remainder of this paper is organized as follows. We recall some basic notions and results related to lattices and t-norms on a bounded lattice, and briefly review the progress in the study of ordinal sums of t-norms on a bounded lattice in Section~\ref{Sec-2}. Section~\ref{Sec-3} is devoted to proposing an alternative definition of ordinal sum of t-norms on a bounded lattice and proving it to be a t-norm, while Section~\ref{Sec-4} shows some examples fitting in the newly proposed ordinal sum of t-norms on a bounded lattice. We end with some conclusions and future work in Section~\ref{Sec-5}.

\section{Preliminaries} \label{Sec-2}

In this section, we recall some basic notions and results related to lattices and t-norms on a bounded lattice, and briefly review the progress in the study of  ordinal sums of t-norms on a bounded lattice.

\subsection{T-norms on a bounded lattice}

A \emph{lattice} \cite{Bir73} is a nonempty set $L$ equipped with a partial order $\leq$ such that any two elements $x$ and $y$ have a greatest lower bound
(called meet or infimum), denoted by $x\wedge y$, as well as a smallest upper bound (called join or supremum), denoted by $x\vee y$. For $a, b\in L$, the symbol $a< b$ means that $a\leq b$ and $a\neq b$. If neither $a\leq b$ nor $b\leq a$, then we say that $a$ and $b$ are incomparable. The set of all elements of $L$ that are incomparable with $a$ is denoted by $I_{a}$. A lattice $(L, \leq, \wedge, \vee)$ is called \emph{bounded} if it has a top element and a bottom element, while it is said to be \emph{complete} if for any $A\subset L$, the greatest lower bound $\bigwedge A$ and the smallest upper bound $\bigvee A$ of $A$ exist. Obviously, any finite lattice is necessarily complete and any complete lattice is necessarily bounded.

Let $(L, \leq, \wedge, \vee)$ be a lattice and $a, b\in L$ with $a\leq b$. The subinterval $[a, b]$ of $L$ is defined as
\[
[a, b]=\{x\in L\ |\ a\leq x\leq b\}\,.
\]
Other subintervals such as $[a, b[$ and $]a, b[$ can be defined similarly.
Obviously, $([a, b], \leq, \wedge, \vee)$ is a bounded lattice with top element $b$ and bottom element~$a$.

\begin{dfnt}\cite{BaeMes99, CooKer94}
Let $(L, \leq, \wedge, \vee)$ be a lattice and $[a, b]$ be a subinterval of $L$. A binary operation
$T\colon [a, b]\times [a, b]\to [a, b]$ is said to be a t-norm on $[a, b]$
if, for any $x, y, z\in [a, b]$, the following conditions are fulfilled:
\begin{itemize}\setlength{\itemindent}{10pt}
\item[(i)] $T(x, y)=T(y, x)$ (commutativity);

\item[(ii)] If $x\leq y$, then $T(x, z)\leq T(y, z)$ (increasingness);

\item[(iii)] $T(T(x, y), z)=T(x, T(y, z))$ (associativity);

\item[(iv)]  $T(b, x)=x$ (neutrality).
\end{itemize}
\end{dfnt}

\begin{thrm}\cite{BaeMes99}
Let $(L, \leq, \wedge, \vee)$ be a lattice, $[a, b]$ be a subinterval of $L$ and $c\in [a, b]$. The binary operation $T_c\colon [a, b]\times [a, b]\to [a, b]$ defined by
\[
T_c(x, y)= \left \{
        \begin {array}{ll}
        x\wedge y                   &\quad \mbox{if $b\in\{x, y\}$}  \\[1mm]
        x\wedge y\wedge c           &\quad \mbox{otherwise,}
        \end {array}
       \right.
\]
is a t-norm on $[a, b]$.
\end{thrm}

If $c=b$ (resp.\ $c=a$), then we retrieve the strongest (resp.\ weakest) t-norm $T_{\wedge}$ (resp.\ $T_{D}$) on  $[a, b]$.

\subsection{Progress in the study of ordinal sums of t-norms on a bounded lattice}

The following result concerning ordinal sum of t-norms on the real unit interval~$[0, 1]$ is well known.

\begin{thrm}\cite{KleMesPap00}
Let $\{\,]a_{i}, b_{i}[\,\}_{i\in I}$ be a family of (nonempty and) pairwisely disjoint open subintervals of $[0, 1]$ and $\{T_{i}\}_{i\in I}$ be a family of t-norms on $[0, 1]$. Then the binary operation $T=\{\langle a_{i}, b_{i}, T_{i}\rangle\}_{i\in I}\colon [0, 1]\times [0, 1]\to [0, 1]$, called the ordinal sum of $\{T_{i}\}_{i\in I}$, defined by
\[
T(x, y)=\left \{
        \begin {array}{ll}
             a_{i}+(b_{i}-a_{i})T_{i}\left(\dfrac{x-a_{i}}{b_{i}-a_{i}}, \dfrac{y-a_{i}}{b_{i}-a_{i}}\right)
                  &\quad \mbox{if $(x, y)\in [a_{i}, b_{i}]^{2}$}  \\[1mm]
             \min\{x, y\}
                  &\quad \mbox{otherwise,}
        \end {array}
       \right.
\]
is a t-norm on $[0, 1]$.
\end{thrm}

\begin{remk}
(i) It is well known that any set consisting of nonempty and pairwisely disjoint open subintervals of the real unit interval~$[0, 1]$ is countable.

(ii) For any $i\in I$, define $\tilde{T}_{i}\colon [a_i, b_i]\times [a_i, b_i]\to [a_i, b_i]$ as follows:
\[\tilde{T}_{i}(x, y)=a_{i}+(b_{i}-a_{i})T_{i}\left(\dfrac{x-a_{i}}{b_{i}-a_{i}}, \dfrac{y-a_{i}}{b_{i}-a_{i}}\right)\,.\]
It is easy to see that $\tilde{T}_{i}$ is a t-norm on $[a_i, b_i]$. So, in the definition of ordinal sum of t-norms on the real unit interval~$[0, 1]$, we can suppose that $\tilde{T}_{i}$ is a t-norm on $[a_i, b_i]$ for any $i\in I$ and replace $a_{i}+(b_{i}-a_{i})T_{i}\left(\dfrac{x-a_{i}}{b_{i}-a_{i}}, \dfrac{y-a_{i}}{b_{i}-a_{i}}\right)$ by $\tilde{T}_{i}(x, y)$. Based on this observation, one can naturally extend the notion of ordinal sum of t-norms from the real unit interval~$[0, 1]$ to a bounded lattice, as Saminger~\cite{Sam06} did in 2006.
\end{remk}

From here on, $(L, \leq, \wedge, \vee, 0, 1)$ denotes a bounded lattice with top element $1$ and bottom element $0$.

\begin{dfnt}\cite{Sam06}
Let $(L, \leq, \wedge, \vee, 0, 1)$ be a bounded lattice, $\{\,]a_{i}, b_{i}[\,\}_{i\in I}$ be a family of pairwisely disjoint subintervals of $L$ and $\{T_{i}\}_{i\in I}$ be a family of t-norms on these subintervals. The ordinal sum $T=\{\langle a_{i}, b_{i}, T_{i}\rangle\}_{i\in I}\colon L\times L\to L$ is given by
\[
T(x, y)=\left \{
        \begin {array}{ll}
             T_{i}(x, y)
                  &\quad \mbox{if $(x, y)\in [a_{i}, b_{i}]^{2}$}  \\[1mm]
             x\wedge y
                  &\quad \mbox{otherwise.}
        \end {array}
       \right.
\]
\end{dfnt}

According to Saminger~\cite{Sam06}, however, the above ordinal sum is not always a t-norm even if there is only one summand, as the following example shows.

\begin{exa}\cite{Sam06}\label{ex2-4}
Consider the complete lattice $L$ with Hasse diagram shown in Figure~\ref{Hasse}.
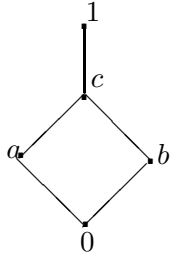
\begin{figure}[t]
\begin{center}
\begin{picture}(150,100)(100,-60)
\put(178,20){$1$}
\put(176,17){$\centerdot$}
\put(177.5,17){\line(0,-1){24}}

\put(180,-5){$c$}
\put(176,-10){$\centerdot$}
\put(178,-7){\line(-1,-1){25}}
\put(152,-32){$\centerdot$}
\put(148,-32){$a$}
\put(178,-8){\line(1,-1){25}}
\put(201,-34){$\centerdot$}
\put(205,-34){$b$}
\put(152,-32){\line(1,-1){25}}
\put(201,-34){\line(-1,-1){24}}
\put(176.5,-58){$\centerdot$}
\put(176,-67){$0$}
\end{picture}
\end{center}
\caption{Hasse diagram of the lattice $L$ in Example~\ref{ex2-4}.}
\label{Hasse}
\end{figure}
The ordinal sum $T=\{\langle a, 1, T_{D}\rangle\}$ given by
Table~\ref{tab-a} is not a t-norm on $L$, since

$T(T(c, c), b)=T(a, b)=0\neq b=T(c, b)=T(c, T(c, b))$.
\begin{table*}
\caption{The ordinal sum $T=\{\langle a, 1, T_{D}\rangle\}$ in Example~\ref{ex2-4}.}
\centering
\begin{tabular}{c|c c c c c c c c c c c }
  \hline
  $T$ & $0$ & $b$ & $a$ & $c$ & $1$\\ \hline
  $0$ & $0$ & $0$ & $0$ & $0$ & $0$\\
  $b$ & $0$ & $b$ & $0$ & $b$ & $b$\\
  $a$ & $0$ & $0$ & $a$ & $a$ & $a$\\
  $c$ & $0$ & $b$ & $a$ & $a$ & $c$\\
  $1$ & $0$ & $b$ & $a$ & $c$ & $1$\\
   \hline
\end{tabular}
\label{tab-a}
\end{table*}
\end{exa}

Several researchers~\cite{Med12,SamKleMes08} characterized when Saminger's ordinal sum of t-norms always leads to a t-norm, while other researchers attempted to modify Saminger's ordinal sum or considered the ordinal sum problem for a particular class of lattices. For instance, Ertu\u{g}rul et al.~\cite{ErtKarMes15} modified Saminger's ordinal sum for one special summand in the following way to make sure it always results in a t-norm.

\begin{thrm}\cite{ErtKarMes15}
Let $(L, \leq, \wedge, \vee, 0, 1)$ be a bounded lattice, $[a, 1]$ be a subinterval of $L$ and $T_{1}$ be a t-norm on $[a, 1]$. Then the binary operation
$T\colon L\times L\to L$ defined by
\begin{equation}\label{eq-tu1}
T(x, y)=\left\{
 \begin {array}{ll}
       T_{1}(x, y)                  &\quad \mbox{if $(x, y)\in [a, 1]^2$}  \\[1mm]
              x\wedge y\wedge a     &\quad \mbox{if $(x, y)\in \left( I_a\times [0, 1[\, \right) \cup \left( [0, 1[\,\times I_a\right)$} \\[1mm]
            x\wedge y               &\quad \mbox{otherwise,}
        \end {array}
       \right.
\end{equation}
is a t-norm on $L$.
\end{thrm}
\begin{remk}
Expression~\eqref{eq-tu1} looks different from the corresponding expression in Theorem~1 of~\cite{ErtKarMes15}, but they are essentially the same.
\end{remk}

\section{An alternative definition of ordinal sum of t-norms on a bounded lattice}\label{Sec-3}

In this section, appropriately dealing with those elements that are incomparable with the endpoints of subintervals, we propose an alternative definition
of ordinal sum of t-norms on a bounded lattice and prove it to always result in a t-norm.

We start by decomposing a bounded lattice with respect to a countable chain, which is crucial in our definition of ordinal sum of t-norms.
Let $(L, \leq, \wedge, \vee, 0, 1)$ be a bounded lattice and $\{c_{i}\}_{i\in \Z}\subseteq L$ be given with $c_{i}\leq c_{i+1}$,
where $\Z$ is the set of all integers. Then $L=S_{1}\cup S_{2}$ and $S_{1}\cap S_{2}=\emptyset$, where
\[
S_{1}=\{x\in L\mid (\exists i\in \Z) (x\in I_{c_{i}})\}=\bigcup_{i\in \Z} I_{c_{i}}
\]
and
\[
S_{2}=\{x\in L \mid (\forall i\in \Z)(x\notin I_{c_{i}})\}=\bigcap_{i\in \Z} L\setminus I_{c_{i}}=L\setminus\bigcup_{i\in \Z} I_{c_{i}}\,.
\]
Further, $S_{1}=A_{1}\cup A_{2}$ and $A_{1}\cap A_{2}=\emptyset$, where
\[
A_{1}=\{x\in S_{1} \mid \inf\{i\in \Z \mid x\in I_{c_{i}}\}=-\infty\}
\]
and
\[
A_{2}=\{x\in S_{1} \mid \inf\{i\in  \Z \mid x\in I_{c_{i}}\}\in \Z\}\,.
\]
It is not difficult to prove that
\[
A_{2}=\bigcup_{i\in \Z} A_{2}^i\,,
\]
where $A_{2}^i=\,]c_{i-1}, 1[\,\cap I_{c_{i}}$.

We furthermore divide $S_{2}$ into three subsets $B_{1}$, $B_{2}$ and $B_{3}$, i.e.,
$S_{2}=B_{1}\cup B_{2}\cup B_{3}$, where
\arraycolsep=2pt
\begin{eqnarray*}
B_{1}&=&\{x\in S_{2} \mid (\forall i\in \Z)(x \geq c_{i})\} =\bigcap_{i\in \Z} [c_{i}, 1]\,, \\[1mm]
B_{2}&=& \{x\in S_{2} \mid (\forall i\in \Z)(x\leq c_{i}) \} =\bigcap_{i\in \Z} [0, c_{i}]
\end{eqnarray*}
and
\[
B_{3}=\{x\in S_{2} \mid (\exists i\in \Z)(x\in [c_{i-1}, c_{i}])\}=\bigcup_{i\in \Z} [c_{i-1}, c_{i}]\,.
\]
Let us further denote
\arraycolsep=2pt
\begin{eqnarray*}
\Delta_1   &=& \left( A_{1}\times [0, 1[\,\right) \cup \left([0, 1[\,\times A_{1}\right)\\[1mm]
\Delta_2^i &=& \left( A_{2}^i\times [c_{i-1}, 1[\,\right) \cup \left( [c_{i-1}, 1[\,\times A_{2}^i\right)\,.
\end{eqnarray*}

We give an example to illustrate the above decomposition.

\begin{exa}\label{ex-3-new}
Let $L=[0, 1]\times [0, 1]$. Define the partial order $\preceq$ on $L$ componentwisely, i.e.,
\[x=(x^{(1)}, x^{(2)})\preceq y=(y^{(1)}, y^{(2)})
  \Longleftrightarrow x^{(n)}\leq y^{(n)}~  (n=1, 2)\,.\]

The meet $\sqcap$ and the join $\sqcup$ with respect to $\preceq$ are given as follows:
\arraycolsep=2pt
\begin{eqnarray*}
(x^{(1)}, x^{(2)}) \sqcap (y^{(1)}, y^{(2)}) &=&(x^{(1)}\wedge y^{(1)}, x^{(2)}\wedge y^{(2)})\\[1mm]
(x^{(1)}, x^{(2)})\sqcup (y^{(1)}, y^{(2)})  &=&(x^{(1)}\vee y^{(1)}, x^{(2)}\vee y^{(2)})\,.
\end{eqnarray*}
Obviously, $(L, \preceq, \sqcap, \sqcup, (0, 0), (1, 1))$ is a complete lattice.

Define $\{c_{i}\}_{i\in \Z}\subset L$ as follows:
\[c_i=(c_i^{(1)}, c_{i}^{(2)})=(\dfrac{1}{3\pi}\arctan i+\dfrac{1}{2}, \dfrac{1}{3\pi}\arctan i+\dfrac{1}{2})\,.\]
Then $c_{i}\preceq c_{i+1}$ and $\bigwedge\limits_{i\in \Z} c_{i}=(\dfrac{1}{3}, \dfrac{1}{3})$. It is not difficult to see that
\[
A_{1}=\left( [0, \dfrac{1}{3}] \times\, ]\dfrac{1}{3}, 1]\right) \cup \left(\, ]\dfrac{1}{3}, 1] \times [0, \dfrac{1}{3}]\right)\,,\quad A_{2}=\bigcup_{i\in \Z} A_{2}^{i},
\]
 where
 \[A_2^i=\left( [c_{i-1}^{(1)}, c_i^{(1)}[\,\times ]c_i^{(2)}, 1]\right) \cup \left(\,]c_i^{(1)}, 1] \times [c_{i-1}^{(2)}, c_i^{(2)}[\,\right)\,,\]
\[
B_{1}=[\dfrac{2}{3}, 1]\times [\dfrac{2}{3}, 1],\quad B_{2}=[0, \dfrac{1}{3}]\times [0, \dfrac{1}{3}],\, \text{and}\ B_{3}=\bigcup_{i\in \Z}[c_{i-1}, c_{i}],
\]
where
\[
[c_{i-1}, c_{i}]=\left\{(x, y)\ |\ \dfrac{1}{3\pi}\arctan (i-1)+\dfrac{1}{2}\leq x, y\leq \dfrac{1}{3\pi}\arctan i+\dfrac{1}{2}\right\}.
\]
So,
\[
S_{2}=B_{1}\cup B_{2}\cup B_{3},\mbox{~and~}S_{1}=L\setminus S_{2}.
\]
Moreover,
  \[\Delta_1=\Big(A_{1}\times [(0, 0), (1, 1)[\,\Big) \cup \Big( [(0, 0), (1, 1)[\,\times A_{1}\Big)\]
and
 \[\Delta_2^i= \left(A_{2}^i\times [c_{i-1}, (1, 1)[\,\right) \cup \left([c_{i-1}, (1, 1)[\,\times A_{2}^i\right)\,.\]
\end{exa}

We are now ready to propose our definition of ordinal sum of t-norms. First, we consider the case of contiguous subintervals.

\begin{dfnt}\label{dfnt-1}
Let $(L, \leq, \wedge, \vee, 0, 1)$ be a complete lattice, $\{c_{i}\}_{i\in \Z}\subseteq L$ with $c_{i}\leq c_{i+1}$, $c=\bigwedge\limits_{i\in \Z} c_{i}$
 and $\{T_{i}\}_{i\in \Z}$ be a family of t-norms on the subintervals $\{[c_{i-1}, c_{i}]\}_{i\in \Z}$. The ordinal sum
$T=\{\langle c_{i-1}, c_{i}, T_{i}\rangle\}_{i\in \Z}\colon L\times L\to L$ is given by
\begin{equation}\label{eq-os1}
T(x, y) =\left \{
        \begin{array}{ll}
             T_{i}(x, y)
                  &\quad \mbox{if $(x, y)\in [c_{i-1}, c_{i}]^2$}  \\[1mm]
             T_{i}(x\wedge c_{i}, y\wedge c_{i})
                  &\quad \mbox{if $(x, y)\in \Delta_2^i$}\\[1mm]
             x\wedge y\wedge c
             &\quad \mbox{if $(x, y)\in \Delta_1$}\\[1mm]
             x\wedge y
                  &\quad \mbox{otherwise.}
        \end{array}
       \right.
\end{equation}
\end{dfnt}

\begin{remk}
(i) For any $i, j\in \Z$, it holds that
\[\Delta_2^i\cap \Delta_1=[c_{i-1}, c_{i}]^2\cap \Delta_1=[c_{i-1}, c_{i}]^2\cap \Delta_2^j=\emptyset\,.\]
In addition, for any $i, j\in \Z$ with $i\not=j$, it holds that
\[\Delta_2^i\cap \Delta_2^j=\emptyset\,.\]
Hence, the operation in~\eqref{eq-os1} is well defined.

(ii) The completeness requirement for $L$ is only used to ensure the existence of $\bigwedge\limits_{i\in \Z} c_{i}$. We could just suppose that $L$ is complete with respect to meet, but meet-completeness implies join-completeness since $L$ has a top element. The completeness requirement is not needed when there exists $i\in \Z$ such that $c_j=c_i$ for any $j<i$, in particular when dealing with finitely many contiguous subintervals.
\end{remk}

Second, we consider the case of not necessarily contiguous subintervals, whose endpoints form a chain.
Let $(L, \leq, \wedge, \vee, 0, 1)$ be a complete lattice, $\left\{[a_{i}, b_{i}]\right\}_{i\in \Z}$ be a family of subintervals of $L$ with $b_{i}\leq a_{i+1}$, $a=\bigwedge\limits_{i\in \Z} a_{i}$ and $\{T_{i}\}_{i\in \Z}$ be a family of t-norms on these subintervals.

We intend to use \eqref{eq-os1} to define the ordinal sum
$T=\{\langle a_{i}, b_{i}, T_{i}\rangle\}_{i\in \Z}$. The process is divided into three steps.

\noindent {\bf Step 1.}
Define $\{c_{i}\}_{i\in \Z}\subseteq L$ as follows:
\[
c_{2i-1}=a_{i} \quad \mbox{and}\quad c_{2i}=b_{i}\,.
\]
It holds that (1) $c_{i}\leq c_{i+1}$; (2) $\bigwedge\limits_{i\in \Z} c_{i}=\bigwedge\limits_{i\in \Z} a_{i}=a$; (3) $[c_{2i-1}, c_{2i}]=[a_{i}, b_{i}]$; (4) $[c_{2i}, c_{2i+1}]=[b_{i}, a_{i+1}]$.

\noindent {\bf Step 2.}
For any $i\in \Z$, endow $[c_{i-1}, c_{i}]$ with a t-norm $\hat{T}_{i}$ as follows:
\[
\hat{T}_{2i}=T_{i} \quad\mbox{and}\quad \hat{T}_{2i+1}=T_{\wedge}\,.
\]

\noindent {\bf Step 3.}
Define the ordinal sum
$T=\{\langle a_{i}, b_{i}, T_{i}\rangle\}_{i\in \Z}$ as $\{\langle c_{i-1}, c_{i}, \hat{T}_{i}\rangle\}_{i\in \Z}$, i.e.,
\begin{equation}\label{eq-os2}
\hspace{-0.4cm}T(x, y) =\left \{
        \begin {array}{ll}
            \hat{T}_{i}(x, y)                               & \quad\mbox{if $(x, y)\in [c_{i-1}, c_{i}]^2$}  \\[1mm]
            \hat{T}_{i}(x\wedge c_{i}, y\wedge c_{i})  & \quad\mbox{if $(x, y)\in \Delta_2^i$} \\[1mm]
             x\wedge y\wedge a                              &\quad\mbox{if $(x, y)\in \Delta_1$} \\[1mm]
             x\wedge y                                      &\quad\mbox{otherwise.}
        \end {array}
       \right.
\end{equation}
It is routine to check that~\eqref{eq-os2} is the same as~\eqref{eq-os3}:
\begin{equation}\label{eq-os3}
\hspace{-0.2cm}T(x, y) =\left \{
        \begin{array}{ll}
             T_{i}(x, y)                               &\quad \mbox{if $(x, y)\in [a_{i}, b_{i}]^{2}$}  \\[1mm]
             T_{i}(x\wedge b_{i}, y\wedge b_{i})  &\quad \mbox{if $(x, y)\in \Lambda_3^i$}\\
              x\wedge y\wedge a_{i}                    & \quad\mbox{if $(x, y)\in \Lambda_2^i$} \\ [1mm]
              x\wedge y\wedge a                        & \quad\mbox{if $(x, y)\in \Lambda_1$} \\[1mm]
              x\wedge y                                & \quad\mbox{otherwise,}
        \end{array}
       \right.
\end{equation}
where
\arraycolsep=2pt
\begin{eqnarray*}
\Lambda_3^i = \Delta_2^{2i}   &=& \Big( \left( ]a_{i}, 1[\,\cap I_{b_{i}}\right) \times [a_{i}, 1[\,\Big)
                            \cup \Big( [a_{i}, 1[\,\times \left(\,]a_{i}, 1[\,\cap I_{b_{i}}\right)\Big)\,,\\[1mm]
\Lambda_2^i = \Delta_2^{2i-1} &=& \Big(\left(\,]b_{i-1}, 1[\,\cap I_{a_{i}}\right)\times [b_{i-1}, 1[\,\Big)
                            \cup \Big( [b_{i-1}, 1[\,\times \left(\,]b_{i-1}, 1[\,\cap I_{a_{i}}\right)\Big)
\end{eqnarray*}
and
\[\Lambda_1=\Delta_1=\Big(A'_{1}\times [0, 1[\,\Big)\cup \Big([0, 1[\,\times A'_{1}\Big),\]
where
\[
A'_{1}=\{x\in L \mid (\exists i\in \Z)(x\in I_{a_{i}}) \mbox{ and } \inf\{i\in \Z \mid x\in I_{a_{i}}\}=-\infty\}\,.
\]

To summarize, for the case of not necessarily contiguous subintervals, we define the ordinal sum of t-norms as follows.

\begin{dfnt}\label{dfnt-2}
Let $(L, \leq, \wedge, \vee, 0, 1)$ be a complete lattice, $\left\{[a_{i}, b_{i}]\right\}_{i\in \Z}$ be a family of subintervals of $L$ with $b_{i}\leq a_{i+1}$, $a=\bigwedge\limits_{i\in \Z} a_{i}$ and $\{T_{i}\}_{i\in \Z}$ be a family of t-norms on these subintervals. The ordinal sum
$T=\{\langle a_{i}, b_{i}, T_{i}\rangle\}_{i\in \Z}$ is given by~\eqref{eq-os3}.
\end{dfnt}

We are now going to prove our main theorems.

\begin{thrm}\label{thrm-1}
Let $(L, \leq, \wedge, \vee, 0, 1)$ be a complete lattice, $\{c_{i}\}_{i\in \Z}\subseteq L$ with $c_{i}\leq c_{i+1}$,
$c=\bigwedge\limits_{i\in \Z} c_{i}$ and $\{T_{i}\}_{i\in \Z}$ be a family of t-norms on the
subintervals $\{[c_{i-1}, c_{i}]\}_{i\in \Z}$. Then the ordinal sum
$T=\{\langle c_{i-1}, c_{i}, T_{i}\rangle\}_{i\in \Z}\colon L\times L\to L$ given by \eqref{eq-os1}
is a t-norm on $L$.
\end{thrm}

The following observations play a key role in simplifying the proof of Theorem~\ref{thrm-1}.

\begin{obs}\label{obsonS2}
The restriction $T|_{S_2}$ of $T$ to $S_2\times S_2$ is a t-norm on $S_2$, where
$T|_{S_2}\colon S_2\times S_2\to S_2$ is given by
\begin{equation*}
T|_{S_2}(x, y) =\left \{
        \begin{array}{ll}
             T_{i}(x, y)
                  &\quad \mbox{if $(x, y)\in [c_{i-1}, c_{i}]^2$}  \\[1mm]
             x\wedge y
                  &\quad \mbox{otherwise.}
        \end{array}
       \right.
\end{equation*}
\end{obs}
\noindent In fact, $\{0, 1\}\subseteq S_2$ and $(S_2, \leq, \wedge, \vee, 0, 1)$ is a bounded lattice. In this lattice, $I_{c_{i}}=\emptyset$
for all $i\in \Z$. Therefore, it follows from Proposition 5.2 in~\cite{Sam06} that $T|_{S_2}$ is a t-norm on $S_2$.

\begin{obs}\label{obs1}
$T(x, y)=T(x\wedge c, y)$ for any $x\in A_{1}$ and any $y\in [0, 1[$\,.
\end{obs}
\noindent In fact, for any $x\in A_{1}$ and any $y\in [0, 1[$, it holds that $T(x, y)=x\wedge y\wedge c$. Note that $x\wedge c\leq c\leq c_{i}$ for any $i\in \Z$.
If $y\in A_{1}$, then
  \[
  T(x\wedge c, y)=(x\wedge c)\wedge y\wedge c=x\wedge y\wedge c=T(x, y)\,.
  \]
  Otherwise,
  \[
  T(x\wedge c, y)=(x\wedge c)\wedge y=x\wedge y\wedge c=T(x, y)\,.
  \]

\begin{obs}\label{obs2}
$T(x, y)=T(x\wedge c_{i}, y)$ for any $x\in A_{2}^i$ and any $y\in [0, 1[$\,.
\end{obs}
\noindent In fact, if $y\in A_{1}$, then
\[
T(x, y)=x\wedge y\wedge c=(x\wedge c_{i})\wedge y\wedge c=T(x\wedge c_{i}, y)\,.
\]
For $y\notin A_{1}$, we distinguish the following cases:\\
- If $y>c_{i}$, then
\[
T(x, y)= T_{i}(x\wedge c_{i}, y\wedge c_{i})=x\wedge c_{i}=x\wedge c_{i}\wedge y=T(x\wedge c_{i}, y)\,.
\]
- If $y\in [c_{i-1}, c_{i}]$ or $y\in A_{2}^i$, then
\[
T(x, y)=T_{i}(x\wedge c_{i}, y\wedge c_{i})=T(x\wedge c_{i}, y)\,.
\]
- If $y\in A_{2}^j$ for some $j<i$, then
\[
T(x, y)=T_{j}(x\wedge c_{j}, y\wedge c_{j})=T_{j}(x\wedge c_{i}\wedge c_{j}, y\wedge c_{j})=T(x\wedge c_{i}, y)\,.
\]
- If $y\in [c_{j-1}, c_{j}]$ for some $j<i$ or $y\leq c$, then
\[
T(x, y)=x\wedge y=x\wedge c_{i}\wedge y=T(x\wedge c_{i}, y)\,.
\]

Now we are ready to prove Theorem~\ref{thrm-1}.

\noindent {\bf Proof of Theorem~\ref{thrm-1}.}\\
Obviously, $T$ is commutative and $1$ is the neutral element of $T$.
We only need to show that $T$ is increasing and associative.

Increasingness: Let $x, y, z\in L$ with $y\leq z$. We need to prove the following inequality
\begin{equation}\label{monoto}
T(x, y)\leq T(x, z)\,.
\end{equation}

\noindent If $1\in\{x, y, z\}$, then~\eqref{monoto} trivially holds. In the following, we only consider $1\notin\{x, y, z\}$.

\emph{By Observation~\ref{obs1}, we can suppose that $x, y, z\notin A_{1}$}.\newline
In fact, if $x\in A_{1}$, then $x\wedge c\in [0, c]$ and~(\ref{monoto}) is equivalent to
\[
T(x\wedge c, y)\leq T(x\wedge c, z)\,.
\]
If $y\in A_{1}$, then $y\wedge c\in [0, c]$, $y\wedge c\leq z$ and~(\ref{monoto}) is equivalent to
\[
T(x, y\wedge c)\leq T(x, z)\,.
\]
If $z\in A_{1}$, then $z\wedge c\in [0, c]$. Note that $y\leq z$ implies either $y\in A_{1}$ or $y\leq c$.
In both cases, $T(x, y)=T(x, y\wedge c)$ and~(\ref{monoto}) is equivalent to
\[
T(x, y\wedge c)\leq T(x, z\wedge c)\,.
\]

\emph{By Observation~\ref{obs2}, we can also suppose that $x, y, z\notin A_{2}$}.\\
In fact, if $x\in A_{2}$, i.e., there exists $i\in \Z$ such that $x\in A_{2}^i$, then $x\wedge c_{i}\in [c_{i-1}, c_{i}]$
and~(\ref{monoto}) is equivalent to
\[
T(x\wedge c_{i}, y)\leq T(x\wedge c_{i}, z)\,.
\]
If $y\in A_{2}$, i.e., there exists $j\in \Z$ such that $y\in A_{2}^j$, then
$y\wedge c_{j}\in [c_{j-1}, c_{j}]$, $y\wedge c_{j}\leq z$ and~(\ref{monoto}) is equivalent to
\[
T(x, y\wedge c_{j})\leq T(x, z)\,.
\]
If $z\in A_{2}$, i.e., there exists $k\in \Z$ such that $z\in A_{2}^k$, then
$z\wedge c_{k}\in [c_{k-1}, c_{k}]$. Note that $y\leq z$, we distinguish the following cases:\newline
- If $y\in A_{2}^j$ for some $j\leq k$, then $y\wedge c_{j}\leq z\wedge c_{k}$ and~(\ref{monoto}) is equivalent to
\[
T(x, y\wedge c_{j})\leq T(x, z\wedge c_{k})\,.
\]
- If $y\in [c_{j-1}, c_{j}]$ for some $j\leq k$ or $y\in [0, c]$, then
$y=y\wedge c_{k}\leq z\wedge c_{k}$ and~(\ref{monoto}) is equivalent to
\[
T(x, y)\leq T(x, z\wedge c_{k})\,.
\]
- If $y\in A_{1}$, then $y\wedge c\leq z\wedge c_{k}$ and~(\ref{monoto}) is equivalent to
\[
T(x, y\wedge c)\leq T(x, z\wedge c_{k})\,.
\]

Based on the discussion above, it suffices to verify that~\eqref{monoto} holds for $x, y, z\in S_{2}=B_{1}\cup B_{2}\cup B_{3}$.
However, in that case the proof follows from Observation~\ref{obsonS2}.

Associativity: Let $x, y, z\in L$. We need to prove the following equality
\begin{equation}\label{asso}
T(T(x, y), z)=T(x, T(y, z))\,.
\end{equation}
If $1\in\{x, y, z\}$, then \eqref{asso} trivially holds. We only consider $1\notin\{x, y, z\}$.

In a similar way as in the case of the increasingness property, we can prove that it suffices to consider $x, y, z\in S_{2}$,
in which case the proof follows from Observation~\ref{obsonS2}.

To conclude, we have proved that $T$ is commutative, increasing, associative and has neutral element $1$, i.e., $T$ is a t-norm on $L$.~$\Box$

\begin{thrm}\label{mainth}
Let $(L, \leq, \wedge, \vee, 0, 1)$ be a complete lattice, $\left\{[a_{i}, b_{i}]\right\}_{i\in \Z}$ be a family of subintervals of $L$ with $b_{i}\leq a_{i+1}$, $a=\bigwedge\limits_{i\in \Z} a_{i}$ and $\{T_{i}\}_{i\in \Z}$ be a family of t-norms on these subintervals. Then the ordinal sum $T=\{\langle a_{i}, b_{i}, T_{i}\rangle\}_{i\in \Z}$ given by~\eqref{eq-os3}
is a t-norm on $L$.
\end{thrm}

\begin{proof}
It follows from Theorem~\ref{thrm-1} and the fact that~\eqref{eq-os3} is actually deduced from~\eqref{eq-os1}.
\end{proof}

Theorem~\ref{mainth} also applies to a finite sequence of subintervals $\left\{[a_{i}, b_{i}]\right\}^{n}_{i=1}$ on a bounded lattice $L$
(\emph{in this case $L$ need not be complete}).
To this end, it suffices to let $a_{i}= b_{i}=a_{1}$ for any $i\in \Z$ with $i<1$
and $b_{i}=a_{i}=b_{n}$ for any $i\in \Z$ with $i>n$.

In the finite case, it holds that
\[\Lambda_1=\left( I_{a_1}\times [0, 1[\,\right)  \cup \left( [0, 1[\,\times I_{a_1}\right) \,.\]

\begin{thrm}
Let $(L, \leq, \wedge, \vee, 0, 1)$ be a bounded lattice, $\left\{[a_{i}, b_{i}]\right\}^{n}_{i=1}$ be a finite sequence of subintervals on $L$ with $b_{i}\leq a_{i+1}$ and $\{T_{i}\}^{n}_{i=1}$ be a finite sequence of t-norms on these subintervals. Then the ordinal sum $T=\{\langle a_{i}, b_{i}, T_{i}\rangle\}^{n}_{i=1}\colon L\times L\to L$ defined by
\begin{equation}\label{eq-osfin}
\hspace{-0.2cm}T(x, y) =\left \{
        \begin{array}{ll}
             T_{i}(x, y)                               &\quad \mbox{if $(x, y)\in [a_{i}, b_{i}]^{2}$}  \\[1mm]
             T_{i}(x\wedge b_{i}, y\wedge b_{i})       & \quad\mbox{if $(x, y)\in \Lambda_3^i$}\\
              x\wedge y\wedge a_{i}                    &\quad \mbox{if $(x, y)\in \Lambda_2^i$} \\ [1mm]
              x\wedge y                                &\quad \mbox{otherwise,}
        \end{array}
       \right.
\end{equation}
is a t-norm on $L$, where $\Lambda_2^1\triangleq \Lambda_1$.
\end{thrm}

Setting $n=1$, we get the ordinal sum with one summand.

\begin{thrm}
Let $(L, \leq, \wedge, \vee, 0, 1)$ be a bounded lattice, $[a, b]$ be a subinterval of $L$ and $T_{1}$ be a t-norm on $[a, b]$. Then the ordinal sum $T=\{\langle a, b, T_{1}\rangle\}\colon L\times L\to L$ defined by
\begin{equation}\label{eq-osone}
\hspace{-0.2cm}T(x, y) =\left \{
        \begin{array}{ll}
             T_{1}(x, y)
                  & \quad\mbox{if $(x, y)\in [a, b]^{2}$}  \\[1mm]
             T_{1}(x\wedge b, y\wedge b)
                 &\quad \mbox{if $(x, y)\in \Lambda$}\\
              x\wedge y\wedge a
                  &\quad \mbox{if $(x, y)\in \left( I_{a}\times [0, 1[\,\right) \cup \left( [0, 1[\,\times I_{a}\right)$} \\ [1mm]
              x\wedge y
                  &\quad \mbox{otherwise,}
        \end{array}
       \right.
\end{equation}
is a t-norm on $L$, where
\[\Lambda=\Big(\left(\,]a, 1[\,\cap I_{b}\right)\times [a, 1[\,\Big) \cup \Big( [a, 1[\,\times \left(\,]a, 1[\,\cap I_{b}\right) \Big)\,.\]
\end{thrm}

Setting $b=1$, \eqref{eq-osone} reduces to \eqref{eq-tu1}.

To conclude this section, we give an example to show that, in our definition of ordinal sum, the condition that the endpoints of the subintervals constitute a chain is indispensable since it assures the well-definedness of our ordinal sum.

\begin{exa}\label{ex3-3}
Consider the complete lattice $L$ with Hasse diagram shown in Figure~\ref{Hasse-31}.
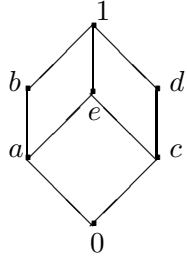
\begin{figure}[h]
\begin{center}
\begin{picture}(150,100)(100,-60)
\put(178,20){$1$}
\put(176,17){$\centerdot$}
\put(177,18){\line(-1,-1){25}}
\put(151,-7){$\centerdot$}
\put(145,-7){$b$}
\put(152,-7){\line(0,-1){25}}
\put(151,-33){$\centerdot$}
\put(145,-32){$a$}
\put(152,-33){\line(1,1){25}}
\put(152,-33){\line(1,-1){25}}
\put(177,18){\line(0,-1){25}}
\put(175.5,-8){$\centerdot$}
\put(175,-18){$e$}
\put(177,18){\line(1,-1){25}}
\put(200,-7){$\centerdot$}
\put(206,-7){$d$}
\put(201,-7){\line(0,-1){25}}
\put(200,-33){$\centerdot$}
\put(206,-32){$c$}
\put(201,-33){\line(-1,1){25}}
\put(201,-33){\line(-1,-1){25}}
\put(176,-58){$\centerdot$}
\put(176,-68){$0$}
\end{picture}
\end{center}
\caption{Hasse diagram of the lattice $L$ in Example~\ref{ex3-3}.}
\label{Hasse-31}
\end{figure}
Let $T_{1}$ be a t-norm on $[a, b]$ and $T_{2}$ be a t-norm on $[c, d]$ (note that these t-norms are unique and coincide with $\land$).
Note that both $e\in\, ]a, 1[\,\cap I_{b}$ and $e\in\, ]c, 1[\,\cap I_{d}$. For the ordinal sum $T=\{\langle a, b, T_{1}\rangle, \langle c, d, T_{2}\rangle\}$
defined by \eqref{eq-osfin}, we have both $T(e, e)=T_{1}(e\wedge b, e\wedge b)=a$ and
$T(e, e)=T_{2}(e\wedge d, e\wedge d)=c$. Therefore, $T$ is not well defined.
\end{exa}

\section{Examples}\label{Sec-4}

In this section, we present two examples that fit in our proposed ordinal sum of t-norms on a bounded lattice.

\begin{exa}\label{ex4-1}
Consider the bounded lattice $L$ with Hasse diagram shown in Figure~\ref{Hasse-2}.
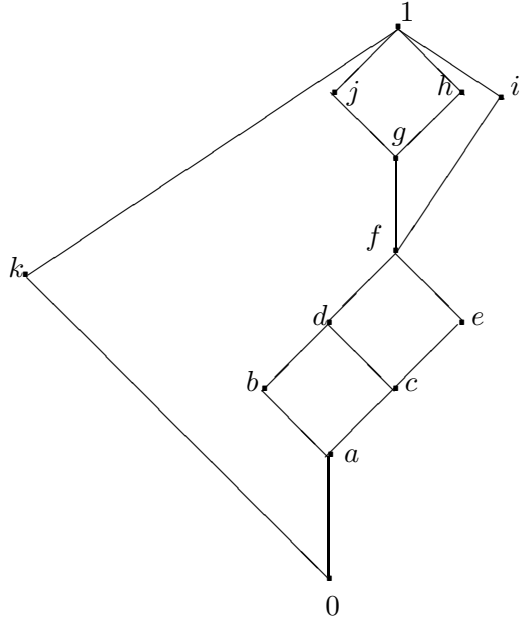
\begin{figure}[h]
\begin{center}
\begin{picture}(150,100)(100,-60)
\put(178,20){$1$}
\put(176,17){$\centerdot$}
\put(177.5,17){\line(-3,-2){140}}
\put(37,-77){\line(1,-1){114}}
\put(35,-77){$\centerdot$}
\put(30, -77){$k$}
\put(177.5,17){\line(-1,-1){25}}
\put(152,-8){$\centerdot$}
\put(158,-10){$j$}
\put(177.5,17){\line(1,-1){25}}
\put(200,-8){$\centerdot$}
\put(192,-8){$h$}
\put(177.5,17){\line(3,-2){40}}
\put(215,-10){$\centerdot$}
\put(220,-8){$i$}
\put(152,-7){\line(1,-1){25}}
\put(201,-7){\line(-1,-1){25}}
\put(175,-33){$\centerdot$}
\put(175,-25){$g$}
\put(176.5,-32){\line(0, -1){35}}
\put(216,-9){\line(-2,-3){39}}
\put(175,-68){$\centerdot$}
\put(165,-65){$f$}
\put(176.5,-68){\line(-1,-1){25}}
\put(200,-95){$\centerdot$}
\put(205,-95){$e$}
\put(176.5,-68){\line(1,-1){25}}
\put(150,-95){$\centerdot$}
\put(145,-95){$d$}
\put(151,-95){\line(1, -1){25}}
\put(200,-95){\line(-1, -1){50}}
\put(175,-120){$\centerdot$}
\put(180,-120){$c$}
\put(151,-95){\line(-1, -1){25}}
\put(125.5,-120){$\centerdot$}
\put(120,-120){$b$}
\put(126,-120){\line(1, -1){25}}
\put(150.5,-145){$\centerdot$}
\put(157,-147){$a$}
\put(151,-145){\line(0,-1){45}}
\put(150,-192){$\centerdot$}
\put(150,-205){$0$}

\end{picture}
\end{center}
\vspace{5cm}
\caption{Hasse diagram of the lattice $L$ in Example~\ref{ex4-1}.}
\label{Hasse-2}
\end{figure}
Consider the ordinal sum $T=\{\langle a, d, T_{c}\rangle, \langle f, h, T_{D}\rangle\}$ defined by \eqref{eq-osfin}. It is routine to check that $T$ (shown in Table~\ref{tab-b}) is a t-norm on $L$.
\begin{table*}
\caption{The ordinal sum $\{\langle a, d, T_{c}\rangle, \langle f, h, T_{D}\rangle\}$ in Example~\ref{ex4-1}. }
\centering
\begin{tabular}{c|c c c c c c c c c c c c c}
  \hline
  $T$ & $0$ & $a$ & $b$ & $c$ & $d$ & $e$ & $f$ & $g$ & $h$ & $i$ & $j$ & $k$ & $1$\\ \hline
  $0$ & $0$ & $0$ & $0$ & $0$ & $0$ & $0$ & $0$ & $0$ & $0$ & $0$ & $0$ & $0$ & $0$\\
  $a$ & $0$ & $a$ & $a$ & $a$ & $a$ & $a$ & $a$ & $a$ & $a$ & $a$ & $a$ & $0$ & $a$\\
  $b$ & $0$ & $a$ & $a$ & $a$ & $b$ & $a$ & $b$ & $b$ & $b$ & $b$ & $b$ & $0$ & $b$\\
  $c$ & $0$ & $a$ & $a$ & $c$ & $c$ & $c$ & $c$ & $c$ & $c$ & $c$ & $c$ & $0$ & $c$\\
  $d$ & $0$ & $a$ & $b$ & $c$ & $d$ & $c$ & $d$ & $d$ & $d$ & $d$ & $d$ & $0$ & $d$\\
  $e$ & $0$ & $a$ & $a$ & $c$ & $c$ & $c$ & $e$ & $e$ & $e$ & $e$ & $e$ & $0$ & $e$ \\
  $f$ & $0$ & $a$ & $b$ & $c$ & $d$ & $e$ & $f$ & $f$ & $f$ & $f$ & $f$ & $0$ & $f$\\
  $g$ & $0$ & $a$ & $b$ & $c$ & $d$ & $e$ & $f$ & $f$ & $g$ & $f$ & $f$ & $0$ & $g$ \\
  $h$ & $0$ & $a$ & $b$ & $c$ & $d$ & $e$ & $f$ & $g$ & $h$ & $f$ & $g$ & $0$ & $h$ \\
  $i$ & $0$ & $a$ & $b$ & $c$ & $d$ & $e$ & $f$ & $f$ & $f$ & $f$ & $f$ & $0$ & $i$ \\
  $j$ & $0$ & $a$ & $b$ & $c$ & $d$ & $e$ & $f$ & $f$ & $g$ & $f$ & $f$ & $0$ & $j$\\
  $k$ & $0$ & $0$ & $0$ & $0$ & $0$ & $0$ & $0$ & $0$ & $0$ & $0$ & $0$ & $0$ & $k$\\
  $1$ & $0$ & $a$ & $b$ & $c$ & $d$ & $e$ & $f$ & $g$ & $h$ & $i$ & $j$ & $k$ & $1$\\
  \hline
\end{tabular}
\label{tab-b}
\end{table*}
\end{exa}

\begin{exa}
Consider the complete lattice $(L, \preceq, \sqcap, \sqcup, (0, 0), (1, 1))$ and the chain $\{c_{i}\}_{i\in \Z}$ introduced
in Example~\ref{ex-3-new}. For any $i\in \Z$, consider the t-norm $T_{i}=\hat{T}_{i}\times \hat{T}_{i}$ (see \cite{BaeMes99})
on $[c_{i-1}, c_{i}]$ given by
\[
T_{i}((x^{(1)}, x^{(2)}), (y^{(1)}, y^{(2)}))=\Big(\hat{T}_{i}(x^{(1)}, y^{(1)}), \hat{T}_{i}(x^{(2)}, y^{(2)})\Big)\,,
\]
where $\hat{T}_{i}$ is a t-norm on $[\dfrac{1}{3\pi}\arctan (i-1)+\dfrac{1}{2}, \dfrac{1}{3\pi}\arctan i+\dfrac{1}{2}]$.

The ordinal sum $T=\{\langle c_{i-1}, c_{i}, T_{i}\rangle\}_{i\in \Z}\colon L\times L\to L$ defined by \eqref{eq-os1} is given by
$T((x^{(1)}, x^{(2)}), (y^{(1)}, y^{(2)}))=(T^{(1)}(x^{(1)}, y^{(1)}), T^{(2)}(x^{(2)}, y^{(2)}))$, where
\[
T^{(1)}(x^{(1)}, y^{(1)}) =\left \{
        \begin {array}{ll}
            \hat{T}_{i}(x^{(1)}, y^{(1)})
                 &\quad\mbox{if $(x, y)\in [c_{i-1}, c_{i}]^{2}$}  \\
            \hat{T}_{i}(x^{(1)}\wedge c_i^{(1)}, y^{(1)}\wedge c_i^{(1)})
                 &\quad\mbox{if $(x, y)\in\Delta_{2}^i$}\\
            x^{(1)}\wedge y^{(1)}\wedge \dfrac{1}{3}
                 &\quad\mbox{if $(x, y)\in\Delta_{1}$}  \\
            x^{(1)}\wedge y^{(1)}
                 &\quad\mbox{otherwise}
        \end {array}
       \right.
\]
and
\[
T^{(2)}(x^{(2)}, y^{(2)}) =\left \{
        \begin {array}{ll}
           \hat{T}_{i}(x^{(2)}, y^{(2)})
                 &\quad\mbox{if $(x, y)\in [c_{i-1}, c_{i}]^{2}$}  \\
             \hat{T}_{i}(x^{(2)}\wedge c_i^{(2)}, y^{(2)}\wedge c_i^{(2)})
                 &\quad\mbox{if $(x, y)\in\Delta_{2}^i$}\\
            x^{(2)}\wedge y^{(2)}\wedge \dfrac{1}{3}
                 &\quad\mbox{if $(x, y)\in\Delta_{1}$}  \\
            x^{(2)}\wedge y^{(2)}
                 &\quad\mbox{otherwise.}
        \end {array}
       \right.
\]
\end{exa}

\section{Conclusion}\label{Sec-5}
In this paper, we have proposed an alternative definition of ordinal sum of countably many t-norms on subintervals of a complete lattice, where the endpoints of the subintervals constitute a chain. The completeness requirement for the lattice is not needed when considering finitely many t-norms. The newly proposed ordinal sum is shown to be always a t-norm. Obviously, our approach can be applied to define the ordinal sum
of triangular conorms on a bounded lattice in a dual way.

Note that we have only partially solved the ordinal sum problem. For future work, it is interesting to consider how to define the ordinal sum of t-norms on subintervals of a bounded lattice in case the endpoints of the subintervals do not constitute a chain (see Example~\ref{ex3-3}).

\section*{Acknowledgement}
This work has been supported in part by National Natural Science Foundation of China (No. 11571106 and No. 11571006), Natural Science Foundation of Zhejiang Province (No. LY20A010006) and NUPTSF (No. NY220029).

\end{document}